\documentclass[a4paper,oneside]{amsart}
\usepackage[english]{babel}
\usepackage{amsmath,amsthm,amssymb,amsfonts}
\usepackage{enumitem}
\usepackage{hyperref}

\newtheorem{thm}{Theorem}
\newtheorem{prp}[thm]{Proposition}
\newtheorem{cor}[thm]{Corollary}
\newtheorem{lem}[thm]{Lemma}

\newcommand{\sloc}{\mathrm{sloc}}
\newcommand{\gen}{\mathrm{gen}}

\newcommand{\cutoff}{\mathcal{C}}

\newcommand{\thr}{\varsigma}

\DeclareMathOperator{\srdet}{det^{1/2}}
\mathchardef\negative="2200
\DeclareMathOperator{\isrdet}{det^{\negative 1/2}}

\DeclareMathOperator{\Kern}{\mathcal{K}}
\DeclareMathOperator{\supp}{\mathrm{supp}}

\newcommand{\opL}{\mathcal{L}}
\newcommand{\BB}{\mathcal{B}}

\newcommand{\N}{\mathbb{N}}
\newcommand{\Z}{\mathbb{Z}}
\newcommand{\R}{\mathbb{R}}
\newcommand{\C}{\mathbb{C}}

\newcommand{\tc}{\,:\,}

\newcommand{\hS}{\mathrm{S}}
\newcommand{\hT}{\mathrm{T}}

\newcommand{\psign}{\varepsilon}

\newcommand{\chr}{\chi}

\newcommand{\vecU}{\mathbf{U}}

\newcommand{\id}{\mathrm{id}}

\newcommand{\lie}{\mathfrak}
\newcommand{\fst}{{\lie{g}_1}}
\newcommand{\snd}{{\lie{g}_2}}
\newcommand{\ddsnd}{\lie{g}_{2,\mathrm{r}}^*}
\newcommand{\dusnd}{{\lie{g}_2^*}}
\newcommand{\cfst}{{(\lie{g}_{1})_\C}}

\newcommand{\defeq}{\mathrel{:=}}

\title[Spectral multipliers on $2$-step groups]{Spectral multipliers on $2$-step groups: topological versus homogeneous dimension}
\author[A. Martini]{Alessio Martini}
\address[A. Martini]{School of Mathematics \\ University of Birmingham \\ Edgbaston \\ Birmingham \\ B17 0AN \\ United Kingdom}
\email{a.martini@bham.ac.uk}
\author[D. M\"uller]{Detlef M\"uller}
\address[D. M\"uller]{Mathematisches Seminar \\ C.-A.-Universit\"at zu Kiel \\ Ludewig-Meyn-Str.\ 4 \\ D-24118 Kiel \\ Germany}
\email{mueller@math.uni-kiel.de}

\begin{document}
\begin{abstract}
Let $G$ be a $2$-step stratified group of topological dimension $d$ and homogeneous dimension $Q$. Let $\opL$ be a homogeneous sub-Laplacian on $G$. By a theorem due to Christ and to Mauceri and Meda, an operator of the form $F(\opL)$ is of weak type $(1,1)$ and bounded on $L^p(G)$ for all $p \in (1,\infty)$ whenever the multiplier $F$ satisfies a scale-invariant smoothness condition of order $s > Q/2$. It is known that, for several $2$-step groups and sub-Laplacians, the threshold $Q/2$ in the smoothness condition is not sharp and in many cases it is possible to push it down to $d/2$. Here we show that, for all $2$-step groups and sub-Laplacians, the sharp threshold is strictly less than $Q/2$, but not less than $d/2$.
\end{abstract}

\maketitle

\section{Introduction}

Let $\opL = -\Delta$ be the Laplace operator on $\R^d$. 
Since $\opL$ is essentially self-adjoint on $L^2(G)$, a functional calculus for $\opL$ is defined via the spectral theorem and an operator of the form $F(\opL)$ is bounded on $L^2(\R^d)$ whenever the Borel function $F : \R \to \C$ is bounded. The investigation of necessary and sufficient conditions for $F(\opL)$ to be bounded on $L^p$ for some $p \neq 2$ in terms of properties of the ``spectral multiplier'' $F$ is a traditional and very active area of research of harmonic analysis.

Among the classical results, a corollary of the Mihlin--H\"ormander multiplier theorem gives a sufficient condition for the $L^p$-boundedness of $F(\opL)$ in terms of a local scale-invariant Sobolev condition of the form 
\begin{equation}\label{eq:mh}
\| F \|_{L^q_{s,\sloc}} \defeq \sup_{t \geq 0} \| F(t \cdot) \, \eta \|_{L^q_s}< \infty
\end{equation}
for appropriate $q \in [1,\infty]$, $s \in [0,\infty)$; here $\eta \in C^\infty_c((0,\infty))$ is any nonzero cutoff and $L^q_s$ is the $L^q$ Sobolev space of order $s$.

\begin{thm}[Mihlin--H\"ormander]\label{thm:mh}
Let $\opL$ be the Laplace operator on $\R^d$. Suppose that the function $F : \R \to \C$ satisfies $\|F\|_{L^2_{s,\sloc}} < \infty$ for some $s > d/2$. Then the operator $F(\opL)$ is of weak type $(1,1)$ and bounded on $L^p(\R^d)$ for all $p \in (1,\infty)$. Further, the associated operator norms are bounded by multiples of $\|F\|_{L^2_{s,\sloc}}$.
\end{thm}

Actually this result is usually stated by restricting the supremum in \eqref{eq:mh} to $t > 0$. However, with the above definition,
\[
\| F \|_{L^q_{s,\sloc}} \sim |F(0)| + \sup_{t > 0} \| F(t \cdot) \, \eta \|_{L^q_s}
\]
and, since the Laplace operator $\opL$ on $\R^d$ has trivial kernel, the usual statement is recovered by applying Theorem \ref{thm:mh} to the multiplier $F \chr_{\R \setminus \{0\}}$. On the other hand, given that we will also discuss operators with nontrivial kernel, the definition in \eqref{eq:mh} seems more convenient here.

The threshold $d/2$ on the order of smoothness $s$ in Theorem \ref{thm:mh} is sharp. More precisely, if we define the sharp threshold $\thr(\opL)$ as the infimum of the $s \in [0,\infty)$ such that
\begin{equation}\label{eq:mhweak}
\exists C \in (0,\infty) \tc \forall F \in \BB \tc \| F(\opL) \|_{L^{1} \to L^{1,\infty}} \leq C \, \| F\|_{L^2_{s,\sloc}}, 
\end{equation}
where $\BB$ is the set of the bounded Borel functions on $\R$, then $\thr(\opL) = d/2$. This can be seen, e.g., by taking $F(\lambda) = |\lambda|^{i\alpha}$, $\alpha \in \R \setminus \{0\}$ (see, e.g., \cite{cowling_maximal_1979,christ_multipliers_1991,sikora_imaginary_2001}). The same example shows that the threshold $d/2$ remains sharp even if we consider a weaker version of the result, without the weak type $(1,1)$ endpoint and with a stronger assumption on the multiplier, in terms of an $L^\infty$ Sobolev norm. Namely, if $\thr_-(\opL)$ is defined as the infimum of the $s \in [0,\infty)$ such that
\begin{equation} \label{eq:mhlp}
\forall p \in (1,\infty) \tc \exists C \in (0,\infty) \tc \forall F \in \BB \tc \| F(\opL) \|_{L^{p} \to L^{p}} \leq C \, \| F\|_{L^\infty_{s,\sloc}},
\end{equation}
then it is also $\thr_-(\opL) = d/2$. Further, the endpoint result in Theorem \ref{thm:mh} can be strengthened in the case of compactly supported multipliers $F$: if $\thr_+(\opL)$ is the infimum of the $s \in [0,\infty)$ such that
\begin{equation} \label{eq:mhl1}
\exists C \in (0,\infty) \tc \forall F \in \BB_c \tc \sup_{t>0} \| F(t\opL) \|_{L^{1} \to L^{1}} \leq C \, \| F\|_{L^2_{s}},
\end{equation}
where $\BB_c = \{ F \in \BB \tc \supp F \subseteq [-1,1] \}$, then again $\thr_+(\opL) = d/2$. By taking $F(\lambda) = (1-\lambda)_+^\alpha$, this strengthened result yields the sharp range of $\alpha$ (namely, $\alpha>d/2-1$) for which the Bochner--Riesz means of order $\alpha$ are $L^1$-bounded (see, e.g., \cite[p.\ 389]{stein_harmonic_1993}).

Results of this type have been obtained in more general contexts than $\R^d$, particularly when $\opL$ is a second-order self-adjoint elliptic differential operator on a manifold $M$. For instance, when $M$ is a compact manifold, then $\thr(\opL) = d/2$, where $d$ is the dimension of $M$ \cite{seeger_boundedness_1989}. Things can be very different on noncompact manifolds and it may even happen that $\thr(\opL) = \infty$ (see, e.g., \cite{clerc_multipliers_1974,christ_spectral_1996}). However the lower bound $\thr(\opL) \geq d/2$ is always true. In fact, locally, at each point of $M$, $\opL$ ``looks like'' the Laplacian $\opL_0$ on $\R^d$ and one can prove that $\thr(\opL) \geq \thr(\opL_0)$ and $\thr_\pm(\opL) \geq \thr_\pm(\opL_0)$ by a transplantation argument \cite{kenig_divergence_1982}.

Much less is known about sharp thresholds when the ellipticity assumption is weakened. Consider the case of a homogeneous sub-Laplacian $\opL$ on an $m$-step stratified group $G$ of homogeneous dimension $Q$. In other words, $G$ is a simply connected nilpotent Lie group, whose Lie algebra $\lie{g}$ is decomposed as a direct sum $\lie{g} = \bigoplus_{j=1}^m \lie{g}_j$ of linear subspaces, called layers, so that $[\lie{g}_j,\lie{g}_1] = \lie{g}_{j+1}$ for $j=1,\dots,m-1$ and $[\lie{g}_m,\lie{g}_1] = \{0\}$. Moreover $Q = \sum_{j=1}^m j \dim \lie{g}_j$ and $\opL = -\sum_{l=1}^k X_l^2$, where $X_1,\dots,X_k$ are left-invariant vector fields on $G$ that form a basis of the first layer $\lie{g}_1$. The sub-Laplacian $\opL$ is a left-invariant second-order self-adjoint hypoelliptic differential operator on $G$, which is not elliptic unless $m=1$, i.e., unless $G$ is abelian and $\opL$ is a Euclidean Laplacian.

Homogeneous sub-Laplacians on stratified groups have been extensively studied, also because of their role as local models for more general hypoelliptic operators (see, e.g., \cite{rothschild_hypoelliptic_1976,folland_applications_1977,nagel_harmonic_1990,ter_elst_weighted_1998}). Several generalizations of Theorem \ref{thm:mh} to this context have been obtained \cite{de_michele_heisenberg_1979,folland_hardy_1982,de_michele_mulipliers_1987}, culminating in the following result independently proved by Christ \cite{christ_multipliers_1991} and by Mauceri and Meda \cite{mauceri_vectorvalued_1990}.

\begin{thm}[Christ, Mauceri and Meda]\label{thm:cmm}
Let $\opL$ be a homogeneous sub-Laplacian on a stratified group $G$ of homogeneous dimension $Q$. Then $\thr(\opL) \leq Q/2$.
\end{thm}

Note that $Q \geq d$, where $d = \dim \lie{g}$ is the topological dimension of $G$. In fact $Q = d$ if and only if $m=1$. Hence Theorem \ref{thm:cmm} reduces to Theorem \ref{thm:mh} when $G$ is abelian and in this case it is sharp. Note also that $Q$ coincides both with the local dimension (associated to the optimal control distance for $\opL$) and the dimension at infinity (i.e., degree of polynomial growth) of $G$. Therefore, for many purposes, the homogeneous dimension $Q$ of a stratified group $G$ plays the role that the dimension $d$ plays for the Laplace operator on $\R^d$.

Also for these reasons, the threshold $Q/2$ in Theorem \ref{thm:cmm} was expected to be sharp in any case and the discovery of counterexamples came initially as a surprise. Consider the simplest case of nonabelian stratified groups $G$, i.e., the Heisenberg groups, where $m = 2$ and $Q-d = \dim \snd =1$. M\"uller and Stein \cite{mller_spectral_1994} proved that, for all homogeneous sub-Laplacians $\opL$ on Heisenberg groups, $\thr(\opL) = d/2$. Independently Hebisch \cite{hebisch_multiplier_1993} proved that $\thr(\opL) \leq d/2$ on the larger class of groups of Heisenberg type.

After this discovery, in the last twenty years several other improvements to Theorem \ref{thm:cmm} in particular cases have been obtained and the inequality $\thr(\opL) \leq d/2$ has been proved for many classes of $2$-step groups \cite{hebisch_multiplier_1995,martini_joint_2012,martini_n32,martini_heisenbergreiter,martini_further}. However, to the best of our knowledge, the upper bound $\thr(\opL) \leq Q/2$ of Theorem \ref{thm:cmm} has been so far the best available result for arbitrary stratified groups, or even for arbitrary $2$-step stratified groups. Moreover, apart from the abelian case, the lower bound $\thr(\opL) \geq d/2$ has been proved only for the Heisenberg groups.

The result that we present here applies instead to all $2$-step groups and homogeneous sub-Laplacians thereon.

\begin{thm}\label{thm:thresholds}
Let $\opL$ be a homogeneous sub-Laplacian on a $2$-step stratified group $G$ of topological dimension $d$ and homogeneous dimension $Q$. Then
\[
d/2 \leq \thr_-(\opL) \leq \thr(\opL) \leq \thr_+(\opL) < Q/2.
\]
\end{thm}

Note that the intermediate inequalities $\thr_-(\opL) \leq \thr(\opL)$ and $\thr(\opL) \leq \thr_+(\opL)$ follow from standard arguments (the former is a consequence of the Marcinkiewicz interpolation theorem; for the latter, see, e.g., \cite[Theorem 4.6]{martini_joint_2012}).
The extreme inequalities are the ones that need a proof.

The inequality $\thr_-(\opL) \geq d/2$ is obtained by studying operators closely related to the Schr\"odinger propagator $e^{it\opL}$. As we show in Section \ref{section:stationaryphase}, via a Mehler-type formula we can write the convolution kernels of these operators as oscillatory integrals on the dual $\dusnd$ of the second layer and lower bounds for the corresponding operator norms can be obtained via the method of stationary phase. In these respects, our approach is not dissimilar to the one of \cite{mller_spectral_1994}, where stationary phase is used to study the imaginary powers $\opL^{i\alpha}$ of the sub-Laplacian. However the analysis of the oscillatory integrals associated to $\opL^{i\alpha}$ turns out to be quite complicated already on the Heisenberg groups, where $\dusnd$ is $1$-dimensional, and a generalization of the argument of \cite{mller_spectral_1994} to arbitrary $2$-step groups seems very difficult. In comparison, the method presented here is much simpler, when applied to Heisenberg (or even Heisenberg-type) groups, and the greater complexity involved with more general $2$-step groups becomes manageable.

For arbitrary $2$-step groups, the main difficulty in applying stationary phase is showing that the phase function admits nondegenerate critical points. In the case of groups of Heisenberg type, the origin of $\dusnd$ is such a point, but this need not be the case for more general $2$-step groups. Nevertheless, as we show in Section \ref{section:positivity}, the Hessian of the phase function becomes nondegenerate if we move slightly away from the origin in a ``generic'' direction. One of the ingredients of the proof is the fact that certain Hankel determinants of Bernoulli numbers are strictly positive, which in turn is related to properties of the Riemann zeta function.

The inequality $\thr_+(\opL) < Q/2$ is proved in Section \ref{section:improvement}. The proof follows the method developed in \cite{martini_heisenbergreiter,martini_further} to show that $\thr_+(\opL) \leq d/2$ for particular classes of $2$-step groups $G$. Here we show that a suitable variation of the method, using elementary estimates for algebraic functions, can be applied to arbitrary $2$-step groups and sub-Laplacians and always yields an improvement to Theorem \ref{thm:cmm}.

The lower bound in Theorem \ref{thm:thresholds} shows that all the multiplier theorems for homogeneous sub-Laplacians on $2$-step groups with threshold $d/2$ obtained so far \cite{hebisch_multiplier_1993,martini_n32,martini_heisenbergreiter,martini_further} are sharp. Moreover, by transplantation, it gives a lower bound to $\thr(\opL)$ and $\thr_\pm(\opL)$ for all sub-Laplacians $\opL$ on $2$-step sub-Riemannian manifolds and all other operators $\opL$ locally modeled on homogeneous sub-Laplacians on $2$-step groups.

An interesting open question is whether Theorem \ref{thm:thresholds} extends to stratified groups of step $m>2$. Indeed Theorem \ref{thm:thresholds} yields, via transference \cite{berkson_transference_1996}, the lower bound $\thr_-(\opL) \geq (\dim \fst + \dim \snd)/2$ for all homogeneous sub-Laplacians $\opL$ on all stratified groups. Moreover improvements to the upper bound $Q/2$ in Theorem \ref{thm:cmm} are known for particular stratified groups of step $m>2$ \cite{hebisch_multiplier_1995,martini_joint_2012}. However the methods used in the present paper do not apply directly to stratified groups of step higher than $2$ and new methods and ideas appear to be necessary.

\subsection*{Acknowledgments}
The authors wish to thank Marco M.\ Peloso and Fulvio Ricci for helpful discussions on the subject of this work.

The first-named author gratefully acknowledges the support of the Alexander von Humboldt Foundation and of the Deutsche Forschungsgemeinschaft (project MA 5222/2-1) during his stay at the Christian-Albrechts-Universit\"at zu Kiel, where this work was initiated.

\section{The stationary phase argument}\label{section:stationaryphase}

Let $G$ be a $2$-step stratified group and $\lie{g} = \fst \oplus \snd$ the stratification of its Lie algebra; in other words, $[\fst,\fst] = \snd$ and $[\lie{g},\snd]=\{0\}$.
Let $d_1 = \dim\fst$, $d_2 = \dim\snd$, $d = d_1+d_2$ and $Q = d_1+2d_2$. Let $X_1,\dots,X_{d_1}$ be a basis of $\fst$ and let $\opL = -\sum_{l=1}^{d_1} X_l^2$ be the corresponding sub-Laplacian. Let $\langle \cdot,\cdot \rangle$ be the inner product on $\fst$ that turns  $X_1,\dots,X_{d_1}$ into an orthonormal basis.

Let $\dusnd$ be the dual of $\snd$ and define, for all $\mu \in \dusnd$, the skew-symmetric endomorphism $J_\mu$ on $\fst$ by
\begin{equation}\label{eq:endomorphism}
\langle J_\mu x,x' \rangle = \mu([x,x']) \qquad\text{for all $x,x' \in \fst$.}
\end{equation}
Consider the space $\lie{so}(\fst)$ of skew-symmetric linear endomorphisms of $\fst$, endowed with the Hilbert--Schmidt inner product determined by the inner product on $\fst$. Since $[\fst,\fst]=\snd$, the linear map $\mu \mapsto J_\mu$ is injective, so we can define an inner product on $\dusnd$ by pulling back the inner product on $\lie{so}(\fst)$, and endow $\snd$ with the dual inner product.

As usual, we identify $\lie{g}$ with $G$ via exponential coordinates, so the Haar measure on $G$ coincides with the Lebesgue measure on $\lie{g}$. If $f \in L^1(G)$ and $\mu \in \dusnd$, then we denote by $f^\mu$ the $\mu$-section of the partial Fourier transform of $f$ along $\snd$, given by
\[
f^\mu(x) = \int_\snd f(x,u) \,e^{-i \langle \mu,u\rangle} \,du
\]
for all $x \in \fst$.

If $A$ is a left-invariant operator on $L^2(G)$, then we denote by $\Kern_A$ its convolution kernel. Denote for $t>0$ by $p_t = \Kern_{e^{-t\opL}}$ the heat kernel associated to the sub-Laplacian $\opL$. Notice that the family of contraction operators $e^{-t\opL}$, $t>0$, admits an analytic extension $e^{-z\opL}$ for $z$ in the complex right half-plane $\Re z>0$; the corresponding convolution Schwartz kernels will be denoted by $p_z$.

Let $\hT$ and $\hS$ be the even meromorphic functions defined by
\begin{equation}\label{eq:TSdef}
\hT(z) = \frac{z}{\tan z}, \qquad \hS(z) = \frac{z}{\sin z}, \qquad z \in \C \setminus \{  k \pi \tc 0 \neq k \in \Z \}.
\end{equation}
Note that $J_\mu$ is naturally identified with a skew-symmetric endomorphism of the complexification $\cfst$ of $\fst$, endowed with the corresponding hermitian inner product, and, for all $z \in \C$, $zJ_\mu$ is a normal endomorphism of $\cfst$. Then the following Mehler-type formula holds.

\begin{prp}\label{prp:heat}
For all $z \in \C$ with $\Re z > 0$, and for all $\mu \in \dusnd$,
\begin{equation}\label{eq:heat}
p_z^\mu(x) = \frac{1}{(4\pi z)^{d_1/2}} \, \srdet \hS(z J_\mu) \exp\left(-\frac{1}{4z} \left\langle \hT(zJ_\mu) x,x\right\rangle\right),
\end{equation}
where roots are meant to be determined by the principal branch.
\end{prp}
\begin{proof}
Several instances and variations of this formula can be found in the literature; see, e.g., \cite{hulanicki_distribution_1976,gaveau_principe_1977,mller_analysis_1990,mller_solvability_1996,randall_heat_1996,lust-piquard_simple-minded_2003} and particularly \cite[Corollary (5.5)]{cygan_heat_1979}. Alternatively, for all $\mu \in \dusnd$, one can apply the general formula of \cite[Theorem 5.2]{mller_solvability_2003} (which indeed applies to much wider classes of second order operators than sub-Laplacians) to the symplectic form $\mu([\cdot,\cdot])$ on $(\ker J_\mu)^\perp$ and observe that, on $\ker J_\mu$, $\mu$-twisted convolution reduces to Euclidean convolution and the heat kernel reduces to the Euclidean heat kernel.
\end{proof}

For all  finite-dimensional normed vector spaces $V$, for all $v \in V$, and for all $\epsilon > 0$, denote by $\cutoff_V(v,\epsilon)$ the set of the smooth functions $\chi : V \to \R$ whose support is contained in the closed ball of center $v$ and radius $\epsilon$.

For all $\chi \in C^\infty_c(\R)$ and $t \in \R$, define $m^\chi_t : \R \to \C$ by
\begin{equation}\label{eq:multiplier}
m^\chi_t(\lambda) 
= \int_\R \chi(s) \, e^{i(t-s)\lambda} \,ds = e^{it\lambda} \, \hat \chi(\lambda),
\end{equation}
where $\hat\chi$ is the Fourier transform of $\chi$. In particular $m^\chi_t$ is in the Schwartz class and moreover, for all $t \in \R$ and $\alpha \geq 0$,
\begin{equation}\label{eq:sloc}
\|m^\chi_t\|_{L^\infty_{\alpha,\sloc}} \leq C_{\alpha,\chi} \, (1+|t|)^\alpha.
\end{equation}

Choose orthonormal coordinates $(u_{d_1},\dots,u_{d_2})$ on $\snd$ and let
\[
\vecU = (-i\partial_{u_1},\dots,-i\partial_{u_{d_2}})
\]
be the corresponding vector of central derivatives on $G$. For all $\chi \in C^\infty_c(\R)$, $\theta \in C^\infty_c(\dusnd)$, and $t \in \R$, define $\Omega^{\chi,\theta}_t$ by
\begin{equation}\label{eq:Q}
\Omega^{\chi,\theta}_t = \Kern_{m^\chi_t(\opL) \, \theta(t\vecU)}.
\end{equation}

\begin{prp}\label{prp:oscillatingformula}
For all $\chi \in \cutoff_\R(0,1/2)$, $\theta \in \cutoff_\dusnd(0,1)$, $t \geq 1$, $y \in \fst$, $v \in \snd$,
\begin{equation}\label{eq:qformula}
\Omega^{\chi,\theta}_t(2ty,t^2v) = t^{-Q/2} \frac{e^{i\pi d_1/4}}{(4\pi)^{d_1/2} (2\pi)^{d_2}} \int_{\R} \chi(s) \, I^\theta(t,s,t^{-1},y,v) \,ds,
\end{equation}
where
\[
I^\theta(t,s,r,y,v) = \int_{\dusnd} \exp(it \Phi(y,v,\mu) -is \Sigma(y,\mu) -ir R(s,r,y,\mu)) \, B(sr,\mu) \, \theta(\mu)  \,d\mu
\]
and
\begin{align*}
\Phi(y,v,\mu) &= -\langle \hT(i J_\mu) y,y \rangle + \langle \mu,v \rangle,\\
B(\sigma,\mu) &= (1-\sigma)^{-d_1/2} \, \srdet \hS\left((1-\sigma) \, i J_\mu\right),\\
\Sigma(y,\mu) &= \langle \hS(iJ_\mu)^2 y, y \rangle = |\hS(iJ_\mu) y|^2,\\
R(s,r,y,\mu) &=  s^2 \langle R_0(sr,i J_\mu) y, y \rangle,
\end{align*}
and $R_0$ is the analytic function on $\{ (\sigma,z) \in \C^2 \tc \sigma \neq 1,\, (1-\sigma) z/\pi \notin \Z \setminus \{0\} \}$ defined by the following Maclaurin expansion in $\sigma$:
\begin{equation}\label{eq:taylor}
(1-\sigma)^{-1} \hT((1-\sigma)z) = \hT(z) + \hS(z)^2 \sigma + R_0(\sigma,z) \, \sigma^2.
\end{equation}
\end{prp}
\begin{proof}
For all $\epsilon > 0$ and $\lambda \in \R$, define $m^\chi_{t,\epsilon}(\lambda) = e^{-\epsilon \lambda} \, m^\chi_{t}(\lambda)$
and let $K_{t,\epsilon} = \Kern_{m_{t,\epsilon}^\chi(\opL)}$, $\Omega^{\chi,\theta}_{t,\epsilon} = \Kern_{m^\chi_{t,\epsilon}(\opL) \, \theta(t\vecU)}$. Then, by \eqref{eq:multiplier}, for all $\epsilon > 0$ and $\mu \in \dusnd$,
\[
K_{t,\epsilon}^\mu = \int_{\R} \chi(s) \, p_{\epsilon-i(t-s)}^\mu \,ds,
\]
and so, for all $x \in \fst$,
\[
(\Omega_{t,\epsilon}^{\chi,\theta})^\mu(x) = K_{t,\epsilon}^\mu(x) \, \theta(t\mu) = \int_{\R} \chi(s) \, p_{\epsilon-i(t-s)}^\mu(x) \, \theta(t\mu) \,ds.
\]

In the last integral, the cutoff $\theta$ gives the localization $|t\mu| \leq 1$ and $\chi$ gives $|s| \leq 1/2$; moreover $t \geq 1$, so $|s/t| \leq 1/2$ and $\|i(t-s) J_\mu\| = |1-s/t| |t\mu| \leq 3/2 < \pi$.
In particular, if we apply formula \eqref{eq:heat} and take the limit as $\epsilon \to 0$, then, by dominated convergence,
\begin{multline*}
(\Omega_t^{\chi,\theta})^\mu(x) = \frac{1}{(4\pi)^{d_1/2}} \int_{\R} \chi(s) \, \exp\left(-\frac{i}{4(t-s)} \left\langle \hT(i(t-s)J_\mu) x,x\right\rangle\right) \\
\times  \frac{e^{i\pi d_1/4}}{(t-s)^{d_1/2}} \, \srdet \hS(i(t-s) J_\mu) \, \theta(t\mu) \,ds.
\end{multline*}

Inversion of the partial Fourier transform and a change of variables gives
\[\begin{split}
\Omega^{\chi,\theta}_t(2ty,t^2v) &= \frac{1}{(4\pi)^{d_1/2} (2\pi)^{d_2}} \int_{\dusnd} \int_{\R} \exp\left(-\frac{i t^2}{t-s} \left\langle \hT(i(t-s)J_\mu) y,y\right\rangle\right) \\
&\quad\times \frac{e^{i\pi d_1/4}}{(t-s)^{d_1/2}} \, \srdet \hS(i(t-s) J_\mu) \, \theta(t\mu) \, \chi(s) \, e^{i\langle \mu,t^2 v\rangle} \,ds \,d\mu \\
&= \frac{e^{i\pi d_1/4}}{(4\pi t)^{d_1/2} (2\pi t)^{d_2}} \int_{\R} \chi(s) \int_{\dusnd} B(s/t,\mu) \, \theta(\mu) \,  \\
&\quad\times \exp\left(-it \left(\frac{1}{1-s/t} \left\langle \hT\left((1-s/t) \, i J_\mu\right) y,y\right\rangle - \langle \mu,v \rangle \right)\right) \, \,d\mu \,ds.
\end{split}\]

It is easily checked that $\frac{\partial}{\partial \sigma} ((1-\sigma)^{-1} \hT((1-\sigma)z))|_{\sigma=0} = \hS(z)^2$, so \eqref{eq:taylor} is indeed a Maclaurin expansion and $R_0$ is well-defined. Moreover \eqref{eq:taylor} yields immediately
\begin{multline*}
t \left(\frac{1}{1-s/t} \left\langle \hT\left((1-s/t) \, i J_\mu\right) y,y\right\rangle - \langle \mu,v \rangle \right) \\= -t \Phi(y,v,\mu) +s \Sigma(y,\mu)  + t^{-1} R(s,t^{-1},y,\mu),
\end{multline*}
where $\Phi$ and $R$ are the functions defined above, and the conclusion follows.
\end{proof}

We are going to use the method of stationary phase to obtain estimates from below of $|\Omega_t^{\chi,\theta}|$. For this we need nondegenerate critical points of the phase function.

\begin{prp}\label{prp:nondegenerate}
Let $\Phi$ be defined as in Proposition \ref{prp:oscillatingformula}. There exist $y_0 \in \fst$, $v_0 \in \snd$, $\mu_0 \in \dusnd$ such that
\begin{equation}\label{eq:nondegcritical}
|\mu_0| < 1, \qquad \nabla_\mu \Phi(y_0,v_0,\mu_0) = 0, \qquad \det \nabla_\mu^2 \Phi(y_0,v_0,\mu_0) \neq 0.
\end{equation}
\end{prp}

The proof of Proposition \ref{prp:nondegenerate} is postponed to the next section. We now see how this fact can be used to obtain the desired estimates.

\begin{prp}\label{prp:stationaryphase}
Let $y_0 \in \fst$, $v_0 \in \snd$, $\mu_0 \in \dusnd$ be satisfying \eqref{eq:nondegcritical}.
Then there exist $\chi \in \cutoff_\R(0,1/2)$, $\theta \in \cutoff_\dusnd(0,1)$, and neighborhoods $U \subseteq \fst$ of $y_0$ and $V \subseteq \snd$ of $v_0$ such that, for all $t \geq 1$, $y \in U$, $v \in V$,
\begin{equation}\label{eq:stationaryphase}
t^{Q-d/2} \Omega^{\chi,\theta}_t(2ty,t^2 v) =e^{i\pi d/4}  e^{it\Psi(y,v)} A^{\chi,\theta}(y,v) + O(t^{-1}),
\end{equation}
where $\Psi, A^{\chi,\theta} \in C^\infty(U \times V)$ are real-valued, $A^{\chi,\theta}(y_0,v_0) \neq 0$,
and $O(t^{-1})$ is uniform in $(y,v) \in U \times V$.
\end{prp}
\begin{proof}
Since $\nabla^2_\mu \Phi(y_0,v_0,\mu_0)$ is nondegenerate and $\nabla_\mu \Phi(y_0,v_0,\mu_0) = 0$, by the implicit function theorem there exist neighborhoods $U_0 \subseteq \fst$ of $y_0$ and $V_0\subseteq \snd$ of $v_0$ such that there is a (unique) smooth function $\mu^c : U_0 \times V_0 \to \dusnd$ such that
\[
\mu^c(y_0,v_0) = \mu_0, \qquad \nabla_\mu \Phi(y,v,\mu^c(y,v)) = 0, \qquad \det \nabla^2_\mu \Phi(y,v,\mu^c(y,v)) \neq 0
\]
for all $(y,v) \in U_0 \times V_0$.

For all sufficiently small $\epsilon \in (0,1-|\mu_0|)$ and all sufficiently small compact neighborhoods $U \subseteq U_0$ of $y_0$ and $V \subseteq V_0$ of $v_0$, if $\theta \in \cutoff_\dusnd(\mu_0,\epsilon)$ and $I^\theta$ is defined as in Proposition \ref{prp:stationaryphase}, then the method of stationary phase \cite[Theorem 7.7.6]{hrmander_analysis_1990} yields
\[\begin{split}
I^\theta&(t,s,r,y,v) \\
&= (2\pi i/t)^{d_2/2} \, \isrdet(\nabla^2_\mu \Phi(y,v,\mu^c(y,v))) \, e^{it\Phi(y,v,\mu^c(y,v))} \theta(\mu^c(y,v)) \\
&\times B(sr,\mu^c(y,v)) \, \exp(-is \Sigma(y,\mu^c(y,v)))\, \exp(-ir R(s,r,y,\mu^c(y,v))) \\
&+ O(t^{-d_2/2 - 1})
\end{split}\]
for all $t \geq 1$, $|s|,|r| < \epsilon$, $y \in U, v \in V$.

Note now that
\[
\exp(-ir R(s,r,y,\mu^c(y,v))) = 1 + O(r)
\]
and
\[
B(\sigma,\mu^c(y,v)) = \srdet \hS\left(i J_{\mu^c(y,v)}\right) + O(\sigma).
\]
Consequently, for all $t > \epsilon^{-1}$, $|s| < \epsilon$, $(y,v) \in U \times V$,
\[\begin{split}
I^\theta&(t,s,t^{-1},y,v) \\
&= (2\pi i/t)^{d_2/2} \, \isrdet (\nabla^2_\mu \Phi(y,v,\mu^c(y,v))) \, e^{it\Phi(y,v,\mu^c(y,v))} \theta(\mu^c(y,v)) \\
&\times \srdet \hS\left(i J_{\mu^c(y,v)}\right) \, \exp(-is \Sigma(y,\mu^c(y,v))) + O(t^{-d_2/2 - 1}).
\end{split}\]
Therefore, by \eqref{eq:qformula}, if $\chi \in \cutoff_\R(0,\epsilon)$, then, for all $t > \epsilon^{-1}$ and $(y,v) \in U \times V$,
\[\begin{split}
t^{Q-d/2} &\Omega^{\chi,\theta}_t(2ty,t^2 v) \\
&=\frac{e^{i\pi d/4}}{(4\pi)^{d_1/2} (2\pi)^{d_2/2}}  e^{it\Phi(y,v,\mu^c(y,v))} \theta(\mu^c(y,v))  \, \srdet \hS\left(i J_{\mu^c(y,v)}\right)\\
&\times  \isrdet (\nabla^2_\mu \Phi(y,v,\mu^c(y,v))) \, \hat\chi(\Sigma(y,\mu^c(y,v))) + O(t^{-1}).
\end{split}\]
By compactness of $U$ and $V$, the last identity is trivial for $1 \leq t \leq \epsilon^{-1}$.
Therefore \eqref{eq:stationaryphase} holds for all $t \geq 1$, $y \in U$, $v \in V$, if we define $A^{\chi,\theta}$ and $\Psi$ by
\begin{align*}
\Psi(y,v) &= \Phi(y,v,\mu^c(y,v)),\\
A^{\chi,\theta}(y,v) &= (4\pi)^{-d_1/2} (2\pi)^{-d_2/2} \srdet \hS\left(i J_{\mu^c(y,v)}\right) \theta(\mu^c(y,v)) \\
&\times \isrdet(\nabla^2_\mu \Phi(y,v,\mu^c(y,v))) \, \hat\chi(\Sigma(y,\mu^c(y,v))).
\end{align*}
In particular
\[\begin{split}
A^{\chi,\theta}(y_0,v_0) &= (4\pi)^{-d_1/2} (2\pi)^{-d_2/2} \srdet \hS\left(i J_{\mu_0}\right) \theta(\mu_0) \\
&\times \isrdet(\nabla^2_\mu \Phi(y_0,v_0,\mu_0)) \, \hat\chi(|\hS(iJ_{\mu_0})y_0|^2).
\end{split}\]
The conclusion follows by choosing $\theta \in \cutoff_\dusnd(\mu_0,\epsilon)$ so that $\theta(\mu_0)\neq 0$, and $\chi \in \cutoff_\R(0,\epsilon)$ so that $\hat\chi$ is real-valued and $\hat\chi(|\hS(iJ_{\mu_0})y_0|^2) \neq 0$; the latter condition is easily satisfied by taking $\chi = \chi_0(\lambda^{-1} \, \cdot)$ for some even $\chi_0 \in \cutoff_\R(0,1)$ with $\hat\chi_0(0) \neq 0$ and $\lambda>0$ sufficiently small.
\end{proof}

\begin{thm}\label{thm:lowerbnd}
There exists $\chi \in \cutoff_\R(0,1/2)$ such that, for all $p \in [1,2]$, there exists $C_{p,\chi} > 0$ such that, for all $t \geq 1$, 
\begin{equation}\label{eq:normestimate}
\| m_t^\chi(\opL) \|_{p \to p} \geq C_{p,\chi} \, t^{d(1/p-1/2)}.
\end{equation}
\end{thm}
\begin{proof}
Let $y_0 \in \fst$, $v_0 \in \snd$, $\mu_0 \in \dusnd$ be given by Proposition \ref{prp:nondegenerate}. Let neighborhoods $U \subseteq \fst$ of $y_0$, $V \subseteq \snd$ of $v_0$, $\chi \in \cutoff_\R(0,1/2)$, $\theta \in \cutoff_\dusnd(0,1)$, $\Psi \in C^\infty(U \times V)$ be given by Proposition~\ref{prp:stationaryphase}.

Note that $m_t^\chi(\opL) \neq 0$ for all $t\geq 1$, because $\chi \neq 0$; hence it is sufficient to prove the estimate \eqref{eq:normestimate} for $t$ large.

Set
\[
F_{t}(u) = (2\pi)^{-d_2} \int_\dusnd \theta(t\mu) \, e^{i \langle \mu, u \rangle} \,d\mu.
\]
Then $F_{t} = t^{-d_2} F_{1}(t^{-1} \cdot)$ and
\[
\Omega^{\chi,\theta}_t = (\delta_0 \otimes F_t) * K_t,
\]
where $K_t = \Kern_{m_t^\chi(\opL)}$ as before.

Choose $\tilde\chi \in \cutoff_\fst(0,c)$, where $c > 0$ is a small parameter to be fixed later. Define
\[
\tilde \Omega_t = (\tilde \chi \otimes \delta_0) * \Omega^{\chi,\theta}_t = (\tilde\chi \otimes F_t) * K_t = m^\chi_t(\opL) (\tilde\chi \otimes F_t).
\]

Note that
\[
\tilde \Omega_t(x,u) = \int_\fst \tilde \chi(x') \, \Omega^{\chi,\theta}_t(x-x',u+[x,x']/2) \,dx',
\]
i.e.,
\[
\tilde \Omega_t(2ty,t^2 v) = \int_\fst \tilde \chi(x') \, \Omega^{\chi,\theta}_t(2t(y-x'/2t),t^2(v+[y,x']/t)) \,dx'.
\]

We would like to apply \eqref{eq:stationaryphase} to estimate $\Omega^{\chi,\theta}_t$ in the above integral and get a lower bound for $|\tilde \Omega_t(2ty,t^2 v)|$ for all sufficiently large $t \geq 1$ and all $y,v$ ranging in sufficiently small neighborhoods of $y_0,v_0$. The problem is that the oscillation coming from the factor $e^{it\Psi}$ could produce cancellation by integrating in $x'$. On the other hand $|\nabla \Psi(y,v)| \lesssim 1$ when $y,v$ range in compact sets. Consequently
\[
e^{it\Psi(y-x'/2t,v+[y,x']/t)} = e^{it\Psi(y,v) + ic \, O(1)}
\]
for all $x' \in \supp \tilde\chi$. By taking $c$ sufficiently small, one obtains that there cannot be too much cancellation (the integrand remains in a convex cone in the complex plane whose aperture is independent of $t$). So from \eqref{eq:stationaryphase} we conclude that there exist sufficiently small neighborhoods $U \subseteq \fst$ of $y_0$ and $V \subseteq \snd$ of $v_0$ such that, for all sufficiently large $t \geq 1$, $y \in U$ and $v \in V$,
\[
|\tilde \Omega_t(2ty,t^2 v)| \gtrsim t^{d/2-Q}.
\]
In particular
\[
\|\tilde \Omega_t\|_p \sim t^{Q/p} \left(\int_G |\tilde \Omega_t(2ty,t^2 v)|^p \,dy \,dv\right)^{1/p} \gtrsim t^{d/2-Q/p'} 
\]
where $p' = p/(p-1)$, while
\[
\| \tilde\chi \otimes F_{t} \|_p \sim t^{-d_2/p'}.
\]
Consequently
\[
\| m^{\chi}_t(\opL) \|_{p \to p} \geq \frac{\|\tilde \Omega_t\|_p}{\| \tilde\chi \otimes F_{t} \|_p} \gtrsim t^{d(1/2-1/p')} = t^{d(1/p-1/2)}
\]
and we are done.
\end{proof}

\begin{cor}
$\thr_-(\opL) \geq d/2$.
\end{cor}
\begin{proof}
This follows by comparison of \eqref{eq:sloc} and \eqref{eq:normestimate}.
\end{proof}

\section{Nondegenerate critical points of the phase function}\label{section:positivity}

This section is devoted to the proof of Proposition \ref{prp:nondegenerate}.

Recall that $\hT$ is defined by \eqref{eq:TSdef}. Hence
\[
\hT(z) = 1 - \sum_{k>0} b_k z^{2k},
\]
where, for all nonzero $k \in \N$,
\begin{equation}\label{eq:bernoulli}
b_k = (-1)^{k-1} 2^{2k} B_{2k} / (2k)! = 2 \pi^{-2k} \zeta(2k),
\end{equation}
the $B_{2k}$ are Bernoulli numbers, and $\zeta$ is the Riemann zeta function (see, e.g., \cite[proof of Theorem 1.2.4]{andrews_special_1999}).
So
\begin{equation}\label{eq:powerseries}
\langle \hT(iJ_\mu) y,y \rangle = |y|^2 - \sum_{k>0} b_k |J_\mu^k y|^2.
\end{equation}
Consequently, if $\Phi$ is defined as in Proposition \ref{prp:oscillatingformula}, then
\[
\Phi(y,v,\mu) = -|y|^2 + \Phi_0(y,\mu) + \langle v,\mu \rangle,
\]
where
\begin{equation}\label{eq:phaseseries}
\Phi_0(y,\mu) = \sum_{k>0} b_k |J_\mu^k y|^2,
\end{equation}
and moreover
\[
\nabla_\mu \Phi(y,v,\mu) = \nabla_\mu \Phi_0(y,\mu) + v, \qquad \nabla^2_\mu \Phi(y,v,\mu) = \nabla^2_\mu \Phi_0(y,\mu).
\]
In particular, the proof of Proposition \ref{prp:nondegenerate} is reduced to showing that there exist $y_0 \in \fst$ and $\mu_0 \in \dusnd$ such that $|\mu_0|<1$ and $\nabla^2_\mu\Phi_0(y_0,\mu_0)$ is nondegenerate, because then by choosing $v_0 = -\nabla_\mu \Phi_0(y_0,\mu_0)$ we have that \eqref{eq:nondegcritical} is satisfied.

Since the linear map $\mu \mapsto J_\mu$ is injective, $\dusnd$ can be identified with the subspace $V$ of $\lie{so}(\fst)$ given by
\[
V = \{ J_\mu \tc \mu \in \dusnd \}.
\]
So in the following we will consider $\Phi_0$ as a function $\fst \times V \to \R$.
Let $V_\gen$ be the homogeneous Zariski-open subset of $V$ whose elements have maximal number of distinct eigenvalues among the elements of $V$.

\begin{lem}\label{lem:genericelements}
Let $S \in V_\gen$ and let $e \in \fst$ be such that the orthogonal projection of $e$ on each eigenspace of $S^2$ is nonzero. Let $N$ be the number of distinct eigenvalues of $S^2$. For all $T \in V$, if
\begin{equation}\label{eq:vanishingassumption}
T S^j e = 0 \qquad\text{for $j =0,\dots,2N-1$,}
\end{equation}
then $T = 0$.
\end{lem}
\begin{proof}
For all $t\in \R$, let $S_t = S + t T$ and let $q_t$ be the minimal polynomial of $S_t^2$.
Since $S_t^2$ is a symmetric linear endomorphism, $q_t$ has no multiple roots and, by definition of $V_\gen$, the degree of $q_t$ is at most $N$.

From \eqref{eq:vanishingassumption} we easily obtain that
\[
S_t^j e = S^j e \qquad\text{for all $j=0,\dots,2N$.} 
\]
In particular
\[
q_t(S^2) e = q_t(S_t^2) e = 0. 
\]
For all eigenvalues $\lambda$ of $S^2$, if $P_\lambda$ is the corresponding spectral projection, then
\[
q_t(\lambda) P_\lambda e = q_t(S^2) P_\lambda e = P_\lambda q_t(S^2) e = 0,
\]
but $P_\lambda e \neq 0$ by assumption and consequently
\[
q_t(\lambda) = 0.
\]
This means that the roots of $q_t$ include all the roots of $q_0$. However $q_0$ has $N$ distinct roots and $q_t$ has degree at most $N$, therefore $q_t = q_0$. In particular
\[
q_0((S+tT)^2) = 0
\]
for all $t \in \R$. By expanding the left-hand side and considering the term that contains the highest power of $t$, one obtains that
\[
T^{2N} = 0
\]
and since $T$ is skew-symmetric this implies that $T = 0$.
\end{proof}

Fix $S \in V_\gen$ and $e \in \fst$ as in Lemma \ref{lem:genericelements}. For all $j \in \N$, define
\[
V_j = \{ T \in V \tc T S^l e = 0 \text{ for } l=0,\dots,j-1\}.
\]
Note that $V_0 = V$. Moreover $V_j \supseteq V_{j+1}$ and, by Lemma \ref{lem:genericelements}, $V_j = \{0\}$ for $j$ sufficiently large.
Let $r \in \N$ be minimal so that $V_{r} = 0$ (note that $r$ may be smaller than the value $2N$ given by Lemma \ref{lem:genericelements}, and in fact $r=1$ if $G$ is of Heisenberg type). Choose a linear complement $W_j$ of $V_{j+1}$ in $V_j$. So $V_j = W_j \oplus V_{j+1}$ and
\begin{equation}\label{eq:directsum}
V = W_0 \oplus \dots \oplus W_{r-1}.
\end{equation}
In addition, for all nonzero $T \in W_j$, $T S^l e = 0$ for $l < j$ but $T S^j e \neq 0$, and in particular the map $W_j \ni T \mapsto T S^j e \in \fst$ is injective.

Let $\Phi_{00}(\mu) = \Phi_0(e,\mu)$ and define, for all sufficiently small $\epsilon > 0$, the bilinear form $H(\epsilon) : V \times V \to \R$ by 
\[
H(\epsilon) = \frac{1}{2} \nabla^2 \Phi_{00}(\epsilon S).
\]
Let moreover $H_{ij}(\epsilon)$ be the restriction of $H(\epsilon)$ to $W_i \times W_j$ for all $i,j=0,\dots,r-1$. If we identify bilinear forms with their representing matrices, then we can think of $H_{ij}(\epsilon)$ as the $(i,j)$-block of $H(\epsilon)$ with respect to the decomposition \eqref{eq:directsum} of $V$. Note that $H(\epsilon)$ is an analytic function of $\epsilon$.

\begin{lem}\label{lem:blockasymptotics}
For all $i,j=0,\dots,r-1$, and all small $\epsilon \in \R$,
\begin{multline*}
H_{ij}(\epsilon) (A,B) \\ = \begin{cases}
O(\epsilon^{i+j+1}) & \text{if $i+j$ is odd,}\\
(-1)^{(i-j)/2} b_{1+(i+j)/2} \epsilon^{i+j} \langle A S^i e, B S^j e \rangle + O(\epsilon^{i+j+2}) &\text{if $i+j$ is even.}
\end{cases}
\end{multline*}
\end{lem}
\begin{proof}
Note that, by \eqref{eq:phaseseries},
\[
\nabla^2 \Phi_{00}(\epsilon S) = \sum_{k>0} b_k \nabla^2 \Phi_k(\epsilon S),
\]
where $\Phi_k(T) = |T^k e|^2$. Moreover the Hessian $\nabla^2 \Phi_k(\epsilon S)(A,B)$ is the bilinear part in $(A,B)$ of the Maclaurin expansion of $\Phi_k(\epsilon S + A + B)$ with respect to $(A,B)$. 

Let $k > 0$. In the expansion of $|(\epsilon S + A + B)^k e|^2$, the bilinear part in $(A,B)$ is
\begin{equation}\label{eq:bilinearpart}
2 (-1)^k \epsilon^{2k-2} \sum_{\alpha+\beta+\gamma = 2k-2}\langle S^\alpha A S^\beta B S^\gamma e, e \rangle.
\end{equation}
If we assume that $A \in W_i$ and $B \in W_j$, then the sum can be restricted to the indices $\alpha,\beta,\gamma$ such that $\alpha \geq i$ and $\gamma \geq j$, because the other summands vanish. In particular the entire sum vanishes unless $2k-2 \geq i+j$.

Hence, if $i+j$ is odd, then \eqref{eq:bilinearpart} vanishes unless $2k-2 \geq i+j+1$, and in particular \eqref{eq:bilinearpart} is $O(\epsilon^{i+j+1})$ for all $k$.

Suppose now that $i+j$ is even. If $2k-2 > i+j$, then $2k-2 \geq i+j+2$ and consequently \eqref{eq:bilinearpart} is $O(\epsilon^{i+j+2})$. If instead $2k-2=i+j$, then it must be $\alpha = i$, $\beta = 0$, and $\gamma = j$, thus \eqref{eq:bilinearpart} can be rewritten as
\[
2 (-1)^{k+i+1} \epsilon^{i+j} \langle B S^j e, A S^i e \rangle
\]
and we are done, because $k+i+1 \equiv (i-j)/2$ modulo $2$.
\end{proof}

The following result completes the proof of Proposition \ref{prp:nondegenerate}.

\begin{prp}
For all sufficiently small $\epsilon \neq 0$, $H(\epsilon)$ is positive definite.
\end{prp}
\begin{proof}
Let $M_\epsilon : V \to V$ be the linear map defined by $M_\epsilon|_{W_j} = (-1)^{\lfloor j/2\rfloor} \epsilon^j \, \id_{W_j}$ for all $j=0,\dots,r-1$. Then, by Lemma \ref{lem:blockasymptotics}, we can write $H(\epsilon)$ as
\[
H(\epsilon)(A,B) = \tilde H(\epsilon)(M_\epsilon A, M_\epsilon B)
\]
for all $A,B \in V$, where
$\tilde H(\epsilon) : V \times V \to \R$ is a bilinear form whose restriction $\tilde H_{ij}(\epsilon)$ to $W_i \times W_j$ satisfies
\[
\tilde H_{ij}(\epsilon)(A,B) \\ = \begin{cases}
O(\epsilon) & \text{if $i+j$ is odd,}\\
b_{1+(i+j)/2} \langle A S^i e, B S^j e \rangle + O(\epsilon^2) &\text{if $i+j$ is even.}
\end{cases}
\]
In particular
\[
\tilde H_{ij}(0)(A,B) \\ = \begin{cases}
0 & \text{if $i+j$ is odd,}\\
b_{1+(i+j)/2} \langle A S^i e, B S^j e \rangle &\text{if $i+j$ is even.}
\end{cases}
\]
We are then reduced to showing that $\tilde H(0)$ is positive definite; in fact, if $\tilde H(0)$ is positive definite, then $\tilde H(\epsilon)$ is positive definite for all sufficiently small $\epsilon \neq 0$ and, for such $\epsilon \neq 0$, $H(\epsilon)$ is positive definite too, because $M_\epsilon$ is invertible.

Since the linear map
\[
V = W_0 \oplus \dots \oplus W_{r-1} \ni (T_0,\dots,T_{r-1}) \mapsto (T_j S^j e)_{j=0}^{r-1} \in \lie{g}_1^r
\]
is injective, we can consider $\tilde H(0)$ as the restriction to a suitable subspace of the bilinear form $K : \lie{g}_1^{r} \times \lie{g}_1^{r} \to \R$ given by
\[
K((v_0,\dots,v_{r-1}),(w_0,\dots,w_{r-1})) = \sum_{i,j=0}^{r-1} c_{ij} \, \langle v_i, w_j \rangle,
\]
where
\[
c_{ij} = \begin{cases}
0 &\text{if $i+j$ is odd,}\\
b_{1+(i+j)/2} &\text{if $i+j$ is even.}
\end{cases}
\]
So it is sufficient to show that $K$ is positive definite.

Let $e_1,\dots,e_{d_1}$ be an orthonormal basis of $\fst$. Then, in the basis
\[
(e_1,0,\dots,0),\dots,(0,\dots,0,e_1),\dots,(e_{d_1},0,\dots,0),\dots,(0,\dots,0,e_{d_1})
\]
of $\lie{g}_1^{r}$, the matrix of $K$ is a $d_1 \times d_1$ block diagonal matrix, whose diagonal blocks are all equal to the matrix $C = (c_{ij})_{i,j=0}^{r-1}$. In other words, the matrix of $K$ is of the form  $C\otimes I_{d_1}$. Hence $K$ is positive definite if and only if $C$ is positive definite.

Since $c_{ij} = 0$ when $i+j$ is odd, one can also consider $C$ as a $2 \times 2$ block diagonal matrix, where the first block is determined by the even rows/columns and the second block by the odd rows/columns. In order to show that $C$ is positive definite, it is sufficient to show the positive definiteness of each diagonal block.

In conclusion, by Sylvester's criterion, we are reduced to showing that matrices of the form
\[
Z_{m,s} = \begin{pmatrix}
b_{m+1} & b_{m+2} & \cdots & b_{m+s} \\
b_{m+2} & b_{m+3} & \cdots & b_{m+s+1} \\
\vdots & \vdots & \ddots & \vdots \\
b_{m+s} & b_{m+s+1} & \cdots & b_{m+2s-1} 
\end{pmatrix}
\]
have positive determinant for all $m,s$. Determinants involving Bernoulli numbers have been studied since long time and explicit formulas for some of them can be found in the literature (see, e.g., \cite{alsalam_determinants_1959,zhang_certain_2014}). However for us it is sufficient to show that the determinant of the above matrices is positive, which can be easily seen by means of properties of the Riemann zeta function $\zeta$.

In fact, from the identity $b_k = 2 \, \pi^{-2k} \zeta(2k)$, we obtain
\[
\det Z_{m,s} = \frac{2^{s}}{\pi^{2s(s+m)}} \det \begin{pmatrix}
\zeta(2m+2) & \zeta(2m+4) & \cdots & \zeta(2m+2s) \\
\zeta(2m+4) & \zeta(2m+6) & \cdots & \zeta(2m+2s+2) \\
\vdots & \vdots & \ddots & \vdots \\
\zeta(2m+2s) & \zeta(2m+2s+2) & \cdots & \zeta(2m+4s-2)
\end{pmatrix}.
\]
If $\mathfrak{S}_s$ denotes the permutation group of $\{1,\dots,s\}$ and $\psign(\sigma)$ denotes the sign of a permutation $\sigma \in \mathfrak{S}_s$, then the last determinant can be rewritten as
\[\begin{split}
&\sum_{\sigma \in \mathfrak{S}_s} \psign(\sigma) \prod_{i=1}^s \zeta(2(i+\sigma(i)+m-1)) \\
&=\frac{1}{s!} \sum_{\sigma,\tau \in \mathfrak{S}_s} \psign(\sigma) \, \psign(\tau) \prod_{i=1}^s \zeta(2(\sigma(i)+\tau(i)+m-1)) \\
&=\frac{1}{s!} \sum_{\sigma,\tau \in \mathfrak{S}_s} \psign(\sigma) \, \psign(\tau) \sum_{k_1=1}^\infty \dots \sum_{k_s=1}^\infty k_1^{-2(\sigma(1)+\tau(1)+m-1)} \cdots k_s^{-2(\sigma(s)+\tau(s)+m-1)} \\
&=\frac{1}{s!} \sum_{k_1=1}^\infty \dots \sum_{k_s=1}^\infty (k_1 \cdots k_s)^{-2(2s+m-1)} \left(\sum_{\sigma \in \mathfrak{S}_s} \psign(\sigma) \, k_1^{2(s-\sigma(1))} \cdots k_s^{2(s-\sigma(s))} \right)^2 \\
&=\frac{1}{s!} \sum_{k_1=1}^\infty \dots \sum_{k_s=1}^\infty (k_1 \cdots k_s)^{-2(2s+m-1)} \prod_{1 \leq i < j \leq s} (k_i^2 -k_j^2)^2,
\end{split}\]
where in the last passage the Vandermonde determinant formula was used. In particular
\[
\det Z_{m,s} = \frac{2^{s}}{s! \, \pi^{2s(s+m)}} \sum_{k_1=1}^\infty \dots \sum_{k_s=1}^\infty (k_1 \cdots k_s)^{-2(2s+m-1)} \prod_{1 \leq i < j \leq s} (k_i^2 -k_j^2)^2 > 0,
\]
and we are done.
\end{proof}

\section{Improvement of the sufficient condition}\label{section:improvement}

We now demonstrate how some estimates obtained in \cite{martini_further}, combined with elementary estimates for multivariate algebraic functions, can be used to obtain an improvement to Theorem \ref{thm:cmm} for all $2$-step groups.

Since the skew-symmetric endomorphism $J_\mu$ defined by \eqref{eq:endomorphism} depends linearly on $\mu$, we can write a spectral decomposition of $\sqrt{-J_\mu^2}$ where eigenvalues and spectral projections are algebraic functions of $\mu$. More precisely, as discussed in \cite[\S 2]{martini_further}, there exist nonzero $M,r_1,\dots,r_M \in \N$ and a nonempty Zariski-open homogeneous subset $\ddsnd \subseteq \dusnd$ such that, for all $\mu \in \ddsnd$, we can write
\begin{equation}\label{eq:Jspectral}
\sqrt{-J_\mu^2} = \sum_{j=1}^M b_j^\mu P_j^\mu
\end{equation}
for distinct $b_1^\mu,\dots,b_M^\mu \in (0,\infty)$ and mutually orthogonal projections $P_1^\mu,\dots,P_M^\mu$ of rank $2r_1,\dots,2r_M$, which are algebraic functions of $\mu$ and are real-analytic for $\mu \in \ddsnd$. Let moreover $P_0^\mu = I - (P_1^\mu+\dots+P_M^\mu)$ be the projection onto $\ker J_\mu$ for all $\mu \in \ddsnd$ and $r_0$ be its rank.

In terms of the eigenvalues $b_j^\mu$ and projections $P_j^\mu$ it is possible to write a fairly explicit formula for the Euclidean Fourier transform $\widehat \Kern_{F(\opL)}$ of the convolution kernel of the operator $F(\opL)$, for all $F \in C^\infty_c(\R)$. Namely, for all $\xi\in \fst$ and $\mu \in \ddsnd$,
\begin{equation}\label{eq:kernelformula}
\widehat \Kern_{F(\opL)}(\xi,\mu) = \sum_{n \in \N^M} F\left(\sum_{j=1}^M (2n_j+r_j)b_j^\mu + |P_0^\mu \xi|^2\right) \prod_{j=1}^M \ell_{n_j}^{(r_j-1)}(|P_j^\mu \xi|^2/b_j^\mu),
\end{equation}
where $\ell_m^{(k)}(t) = 2^{k+1} (-1)^m e^{-t} L_m^{(k)}(2t)$ and $L_m^{(k)}$ is the $m$th Laguerre polynomial of type $k$, for all $t \in \R$ and $m,k \in \N$ (cf.\ \cite[Proposition 6]{martini_further}). As shown in \cite[\S 4]{martini_further}, by means of this formula it is possible to estimate $\mu$-derivatives of $\widehat\Kern_{F(\opL)}(\xi,\mu)$ in terms of expressions analogous to the right-hand side of \eqref{eq:kernelformula}, but involving derivatives of $F$, provided that suitable estimates for $\mu$-derivatives of the $b_j^\mu$ and $P_j^\mu$ hold.

For the reader's convenience, we now state as a lemma a particular case of \cite[Corollary 19]{martini_further}, that will be sufficient for our purpose. For technical reasons, the estimate stated below involves a cutoff in the variable $\mu$.

As in \S \ref{section:stationaryphase} above, fix orthonormal coordinates $(u_1,\dots,u_{d_2})$ on $\snd$ and dual coordinates $(\mu_1,\dots,\mu_{d_2})$ on $\dusnd$, and let $\vecU$ be the corresponding vector of central derivatives on $G$. 
Moreover set $\langle t \rangle = 1+|t|$ for all $t \in \R$.

\begin{lem}\label{lem:derivative_estimate}
Let $D$ be a smooth vector field on $\ddsnd$, thought of as a first-order differential operator in the variable $\mu \in \ddsnd$. 
Suppose that there exists
$\kappa \in (0,\infty)$ such that
\begin{gather*}
|D b_1^\mu| \leq \kappa b_1^\mu, \dots, |D b_M^\mu| \leq \kappa b_M^\mu, \\
\|D P_0^\mu\| \leq \kappa,\dots,\|D P_M^\mu\| \leq \kappa
\end{gather*}
 for all $\mu \in \ddsnd$. 
Then, for all $F \in C^\infty_c(\R)$, all $\chi \in C^\infty_c(\ddsnd)$, 
and all $\alpha \in \{0,1\}$,
\begin{multline*}
\int_{\fst} \Bigl|D^\alpha \widehat\Kern_{F(\opL) \, \chi(\vecU)}(\xi,\mu) \Bigr|^2 \, d\xi 
\leq C_{\kappa,\alpha} \sum_{\iota \in I_{\alpha}} \int_{[0,\infty)} \int_{R_\iota} \sum_{n \in \N^M} |D^{k_\iota} \chi(\mu)|^2 \\
\times \left|F^{(\gamma_\iota)}\left(\sum_{j=1}^M (2(n_j+s_j)+r_j) b_j^\mu  + \eta\right)\right|^2 
\prod_{j=1}^M \left[ (b^\mu_j)^{1+a^\iota_j} \langle n_j \rangle^{a^\iota_j} \right] \,d\nu_{\iota}(s) \,d\sigma_\iota(\eta),
\end{multline*}
for all $\mu \in \ddsnd$, where $I_{\alpha}$ is a finite set and, for all $\iota \in I_{\alpha}$,
\begin{itemize}
\item $a^\iota,\beta^\iota \in \N^M$, $\gamma_\iota,k_\iota \in \N$, $\gamma_\iota+k_\iota \leq \alpha$,
\item $R_\iota = \prod_{j=1}^M [0,\beta^\iota_j]$ and $\nu_\iota$ is a Borel probability measure on $R_\iota$,
\item $\sigma_\iota$ is a regular Borel measure on $[0,\infty)$.
\end{itemize}
\end{lem}

In view of the above lemma, we are interested in estimates for $\mu$-derivatives of $b_j^\mu$ and $P_j^\mu$. Indeed in \cite{martini_further} very precise estimates are obtained for particular classes of $2$-step groups, which eventually lead to proving that $\thr_+(\opL) \leq d/2$ in those cases. Here however a simpler estimate will be sufficient, that holds for all $2$-step groups and sub-Laplacians and comes from the very fact that the $b_j^\mu$ and $P_j^\mu$ are algebraic functions of $\mu$.

\begin{lem}\label{lem:hompoly}
There exists a nonzero homogeneous polynomial $H: \dusnd \to \R$ such that, for all $\mu \in \ddsnd$,
\begin{gather*}
|\partial_{\mu_k} b_j^\mu / b_j^\mu | \leq |\mu|^{h-1} |H(\mu)|^{-1}, \qquad j=1,\dots,M, \, k=1,\dots,d_2,\\
\|\partial_{\mu_k} P_j^\mu \| \leq |\mu|^{h-1} |H(\mu)|^{-1}, \qquad j=0,\dots,M, \, k=1,\dots,d_2,
\end{gather*}
where $h = \deg H$.
\end{lem}
\begin{proof}
Note that the expressions $\partial_{\mu_k} b_j^\mu / b_j^\mu$ and $\partial_{\mu_k} P_j^\mu$ are algebraic functions of $\mu$ and are homogeneous of degree $-1$. The conclusion follows by a simple adaptation of the proof of \cite[Lemma 4.2]{mller_solvability_1996}, taking into consideration the homogeneity (the $\partial_{\mu_k} P_j^\mu$ are matrix-valued functions, but one can argue componentwise).
\end{proof}

We can now combine the two lemmas above to obtain weighted $L^2$-estimates and $L^1$-estimates for $\Kern_{F(\opL)}$ and eventually prove that $\thr_+(\opL) < Q/2$.

\begin{prp}\label{prp:improv}
Let $H$ and $h$ be as in Lemma \ref{lem:hompoly} and set $h_0 = \max\{h,1\}$.
\begin{enumerate}[label=(\roman*),leftmargin=1.5em]
\item\label{en:improv_l2} For all compact sets $K \subseteq \R$, for all $\beta \geq 0$, for all $\alpha \in [0,(2h_0)^{-1})$, for all $s > \beta+\alpha$, if $F : \R \to \C$ is supported in $K$, then
\[
\int_G \Bigl| (1+|x|+|u|^{1/2})^\beta \, (1+|u|)^\alpha \, \Kern_{F(\opL)}(x,u) \Bigr|^2 \, dx \,du \leq C_{K,\alpha,\beta,s} \|F\|_{W_2^{s}}^2.
\]
\item\label{en:improv_l1} For all compact sets $K \subseteq \R$, for all $s > Q/2-(2h_0)^{-1}$, if $F : \R \to \C$ is supported in $K$, then
\[
\|\Kern_{F(\opL)}\|_1 \leq C_{K,s} \|F\|_{W_2^{s}}.
\]
\end{enumerate}
\end{prp}
\begin{proof}
Without loss of generality, we may assume that $F$ is smooth.

Set $\tilde H(\mu) = |\mu|^{-h} |H(\mu)|$.
Fix $k \in \{1,\dots,d_2\}$. Define the first-order differential operator $D$ on $\ddsnd$ by
\[
D = |\mu| \, \tilde H(\mu) \, \partial_{\mu_k}.
\]
By Lemma~\ref{lem:hompoly},
\[
|D^\alpha b_j^\mu| \lesssim b_j^\mu, \qquad \|D^\alpha P_j^\mu\| \lesssim 1
\]
for $\alpha=0,1$. Moreover, by homogeneity considerations,
\[
| D^\alpha |\mu| | \lesssim |\mu|, \qquad | D^\alpha (\tilde H(\mu)) | \lesssim \tilde H(\mu)
\]
for $\alpha=0,1$. In particular, if we choose nonnegative functions $\chi,\tilde\chi \in C^\infty_c((0,\infty))$ such that
\begin{equation}\label{eq:partition}
\sum_{k \in \Z} \chi(2^k \lambda) = 1
\end{equation}
for all $\lambda \in (0,\infty)$ and $\tilde\chi \chi = \chi$, and if we define, for all $\rho,\delta \in (0,\infty)$,
\[
\chi_{\rho,\delta}(\mu) = \chi(\rho^{-1} |\mu|) \, \chi(\delta^{-1} \tilde H(\mu)), \quad \tilde\chi_{\rho,\delta}(\mu) = \tilde\chi(\rho^{-1} |\mu|) \, \tilde\chi(\delta^{-1} \tilde H(\mu)),
\]
then
\[
|D^\alpha \chi_{\rho,\delta}|^2 \lesssim \tilde\chi_{\rho,\delta}
\]
for $\alpha=0,1$, uniformly in $\rho,\delta \in (0,\infty)$.

Therefore Lemma \ref{lem:derivative_estimate}
gives that, for all $\rho,\delta \in (0,\infty)$,
\begin{multline*}
\int_{\fst} \Bigl| D^\alpha \widehat\Kern_{F(\opL) \, \chi_{\rho,\delta}(\vecU)}(\xi,\mu) \Bigr|^2 \, d\xi 
\leq C_{\alpha} \sum_{\iota \in I_{\alpha}} \int_{[ 0,\infty )} \int_{R_\iota} \sum_{n \in \N^M} \tilde\chi_{\rho,\delta}(\mu) \\
\times \left|F^{(\gamma_\iota)}\left(\sum_{j=1}^M (2(n_j+s_j)+r_j) b_j^\mu  + \eta\right)\right|^2 
\prod_{j=1}^M \left[ (b^\mu_j)^{1+a^\iota_j} \langle n_j \rangle^{a^\iota_j} \right] \,d\nu_{\iota}(s) \,d\sigma_\iota(\eta),
\end{multline*}
for $\alpha=0,1$, where  $\widehat\Kern_{F(\opL) \, \chi_{\rho,\delta}(\vecU)}$ is the Euclidean Fourier transform of $\Kern_{F(\opL) \, \chi_{\rho,\delta}(\vecU)}$,  $I_{\alpha}$ is a finite set and, for all $\iota \in I_{\alpha}$,
\begin{itemize}
\item $a^\iota,\beta^\iota \in \N^M$, $\gamma_\iota \in \N$, $\gamma_\iota  \leq \alpha$,
\item $R_\iota = \prod_{j=1}^M [0,\beta^\iota_j]$ and $\nu_\iota$ is a Borel probability measure on $R_\iota$,
\item $\sigma_\iota$ is a regular Borel measure on $[0,\infty)$.
\end{itemize}
If we assume that $\supp F \subseteq K$ for some compact set $K \subseteq \R$, then in the above integral the quantities $b^\mu_j \langle n_j \rangle$ are bounded where the integrand does not vanish. The previous inequality and the Plancherel formula then yield
\begin{multline*}
\int_G \Bigl| u_k^\alpha \, \Kern_{F(\opL) \, \chi_{\rho,\delta}(\vecU)}(x,u) \Bigr|^2 \, dx \,du
\leq C_{K,\alpha}  \sum_{\iota \in I_{\alpha}} \int_{[ 0,\infty )} \int_{R_\iota} \sum_{n \in \N^M} \int_\dusnd \tilde\chi_{\rho,\delta}(\mu) \\
\times \left|F^{(\gamma_\iota)}\left(\sum_{j=1}^M (2(n_j+s_j)+r_j) b_j^\mu  + \eta\right)\right|^2 
(|\mu| \, \tilde H(\mu))^{-2\alpha} \,
\prod_{j=1}^M b^\mu_j  \,d\mu  \,d\nu_{\iota}(s) \,d\sigma_\iota(\eta).
\end{multline*}
Passing to polar coordinates in the inner integral in $\mu$ and rescaling gives
\begin{multline*}
\int_G \Bigl| u_k^\alpha \, \Kern_{F(\opL) \, \chi_{\rho,\delta}(\vecU)}(x,u) \Bigr|^2 \, dx \,du
\leq C_{K,\alpha} \, \rho^{M+d_2-2\alpha} \, \delta^{-2\alpha} \\
\times \sum_{\iota \in I_{\alpha}} \int_{[ 0,\infty )} \int_{R_\iota}  \int_{S_\dusnd} \int_0^\infty \sum_{n \in \N^M} \tilde\chi\left(\frac{\rho^{-1} \lambda}{\sum_{j=1}^M (2(n_j+s_j)+r_j) b_j^\omega}\right) \\
\times \left|F^{(\gamma_\iota)}\left(\lambda   + \eta\right)\right|^2  \,\tilde\chi(\delta^{-1} \tilde H(\omega)) 
\prod_{j=1}^M b^\omega_j  \,\frac{d\lambda}{\lambda} \,d\omega \,d\nu_{\iota}(s) \,d\sigma_\iota(\eta),
\end{multline*}
where $S_\dusnd$ is the unit sphere in $\dusnd$. In the above sum in $n$, the number of nonvanishing summands is bounded by a constant times $(\rho^{-1} \lambda)^M \prod_{j=1}^M (b_j^\omega)^{-1}$, hence
\begin{multline*}
\int_G \Bigl| u_k^\alpha \, \Kern_{F(\opL) \, \chi_{\rho,\delta}(\vecU)}(x,u) \Bigr|^2 \, dx \,du
\leq C_{K,\alpha} \, \rho^{d_2-2\alpha} \, \delta^{-2\alpha} \int_{S_\dusnd} \,\tilde\chi(\delta^{-1} \tilde H(\omega)) \,d\omega \\
\times \sum_{\iota \in I_{\alpha}} \int_{[ 0,\infty )}   \int_0^\infty  
 \left|F^{(\gamma_\iota)}\left(\lambda   + \eta\right)\right|^2  
\lambda^{M-1}\,d\lambda  \,d\sigma_\iota(\eta).
\end{multline*}
By using again the fact that $\supp F \subseteq K$, we finally obtain
\begin{multline}\label{eq:piece_estimate}
\int_G \Bigl| |u|^\alpha \, \Kern_{F(\opL) \, \chi_{\rho,\delta}(\vecU)}(x,u) \Bigr|^2 \, dx \,du \\
\leq C_{K,\alpha} \|F\|^2_{W_2^\alpha} \, \rho^{d_2-2\alpha} \, \delta^{-2\alpha} \int_{S_\dusnd} \,\tilde\chi(\delta^{-1} \tilde H(\omega)) \,d\omega
\end{multline}
for all $\rho,\delta \in (0,\infty)$ and for $\alpha=0,1$. Interpolation then gives the inequality \eqref{eq:piece_estimate} for all $\alpha \in [0,1]$.

Note now that, by \eqref{eq:partition},
\[
F(\opL) = \sum_{\substack{ m \in \Z \\ m \geq m_0}} \sum_{\substack{ l \in \Z \\ l\geq l_0}} F(\opL) \, \chi_{2^{-m},2^{-l}}(\vecU),
\]
where $m_0\in\Z$ depends on the compact set $K$ (cf.\ \cite[proof of Proposition 22]{martini_further}) and $l_0\in \Z$ on $\max_{S_{\dusnd}} \tilde H$. In particular, from \eqref{eq:piece_estimate} and Minkowski's inequality we obtain that, for all $\alpha < \min\{1,d_2/2\}$,
\[
\left(\int_G \Bigl| |u|^\alpha \, \Kern_{F(\opL)}(x,u) \Bigr|^2 \, dx \,du\right)^{1/2}
\leq C_{K,\alpha} \|F\|_{W_2^\alpha} \, \sum_{\substack{l \in \Z \\ l \geq l_0}} 2^{\alpha l} \left(\int_{S_\dusnd} \,\tilde\chi(2^{l} \tilde H(\omega)) \,d\omega\right)^{1/2}.
\]
On the other hand, for all $\epsilon > 0$, by the Cauchy--Schwarz inequality,
\[\begin{split}
\sum_{\substack{l \in \Z \\ l \geq l_0}} 2^{\alpha l}  \left(\int_{S_\dusnd} \,\tilde\chi(2^{l} \tilde H(\omega)) \,d\omega\right)^{1/2} 
&\leq C_{\epsilon} \left(\sum_{l \in \Z} 2^{2(\alpha+\epsilon) l}  \int_{S_\dusnd} \,\tilde\chi(2^{l} \tilde H(\omega)) \,d\omega \right)^{1/2}  \\
&\leq C_{\epsilon} \left( \int_{S_{\dusnd}} \tilde H(\omega)^{-2(\alpha+\epsilon)} \,d\omega \right)^{1/2}.
\end{split}\]
Since $H$ is a homogeneous polynomial of degree $h$, the last integral is finite when $\alpha+\epsilon < (2h)^{-1}$ (see, e.g., \cite[Lemma 2.1]{mller_twisted_1984}). Hence, for all $\alpha < (2h_0)^{-1}$,
\begin{equation}\label{eq:partialweighted}
\left(\int_G \Bigl| |u|^\alpha \, \Kern_{F(\opL)}(x,u) \Bigr|^2 \, dx \,du\right)^{1/2} \leq C_{K,\alpha} \|F\|_{W_2^\alpha}.
\end{equation}
Interpolation of \eqref{eq:partialweighted} with the standard estimate of \cite[Lemma 1.2]{mauceri_vectorvalued_1990} (see, e.g, the proof of \cite[Proposition 12]{martini_heisenbergreiter}) then gives \ref{en:improv_l2}, and \ref{en:improv_l1} follows by H\"older's inequality (see, e.g., the proof of \cite[Proposition 3]{martini_heisenbergreiter}).
\end{proof}

\begin{cor}
If $h_0$ is defined as in Proposition \ref{prp:improv}, then
\[
\thr_+(\opL) \leq Q/2 - 1/(2h_0) < Q/2.
\]
\end{cor}
\begin{proof}
This follows from Proposition \ref{prp:improv}\ref{en:improv_l1} and the fact that, by homogeneity of $\opL$,
\[
\|\Kern_{F(t\opL)}\|_1 = \|\Kern_{F(\opL)}\|_1
\]
for all $t > 0$.
\end{proof}

\end{document}